\documentclass{amsart}
\usepackage{a4wide,amssymb,showlabels}

\newcommand{\A}{\mathcal A}

\newcommand{\K}{\mathcal K}
\newcommand{\w}{\omega}
\newcommand{\id}{\mathrm{id}}
\newcommand{\dom}{\mathrm{dom}}
\newcommand{\Ra}{\Rightarrow}
\newcommand{\supp}{\mathrm{supp}}
\newcommand{\IN}{\mathbb N}
\newcommand{\Tau}{\mathcal T}

\newtheorem{theorem}{Theorem}[section]
\newtheorem{corollary}[theorem]{Corollary}
\newtheorem{proposition}[theorem]{Proposition}
\newtheorem{lemma}[theorem]{Lemma}
\newtheorem{example}[theorem]{Example}
\newtheorem{problem}[theorem]{Problem}
\newtheorem{claim}[theorem]{Claim}

\title[{A metrizable semitopological semilattice with non-closed partial order}]{A metrizable semitopological semilattice\\ with non-closed partial order}
\author{Taras Banakh, Serhii Bardyla, and Alex Ravsky}
\address{T.Banakh: Ivan Franko National University of Lviv (Ukraine) and Jan Kochanowski University in Kielce (Poland)}
\email{t.o.banakh@gmail.com}
\address{S.~Bardyla: Institute of Mathematics, Kurt G\"{o}del Research Center, Vienna (Austria)}
\email{sbardyla@yahoo.com}
\thanks{The second author was supported by the Austrian Science Fund FWF (Grant  I 3709-N35).}
\address{A.Ravsky: Department of Analysis, Geometry and Topology,
Pidstryhach Institute for Applied Problems of Mechanics and Mathematics
National Academy of Sciences of Ukraine}
\email{alexander.ravsky@uni-wuerzburg.de}

\begin{document}
\begin{abstract} We construct a metrizable semitopological semilattice $X$ whose partial order $P=\{(x,y)\in X\times X:xy=x\}$ is a non-closed dense subset of $X\times X$. As a by-product we find necessary and sufficient conditions for the existence of a (metrizable) Hausdorff topology on a set, act, semigroup or semilattice, having a prescribed countable family of convergent sequences.
\end{abstract}
\maketitle

\section{Introduction}

In this paper we shall construct an example of a metrizable semitopological semilattice with non-closed partial order. 

A {\em semilattice} is a commutative semigroup $X$ whose any element $x\in X$ is an {\em idempotent} in the sense that $xx=x$. A typical example of a semilattice is any partially ordered set $X$ in which any finite non-empty set $F\subset X$ has the greatest lower bound $\inf(F)$. In this case the binary operation $X\times X\to X$, $(xy)\mapsto\inf\{x,y\}$, turns $X$ into a semilattice. 

Each semilattice $X$ carries a partial order $\le$ defined by $x\le y$ iff $xy=x$. For this partial order we have $xy=\inf\{x,y\}$.

A ({\em semi\/}){\em topological semilattice} is a semilattice $X$ endowed with a topology such that the binary operation $X\times X\to X$, $xy\mapsto xy$, is (separately) continuous. 

The continuity of the semilattice operation in a Hausdorff topological semilattice implies the following well-known fact, see \cite[VI-1.14]{Bible}.

\begin{proposition} For any Hausdorff topological semilattice $X$ the partial order $P=\{(x,y)\in X\times X:xy=x\}$ is a closed subset of $X\times X$.
\end{proposition}

It is natural to ask whether this proposition remains true for Hausdorff semitopological semigroups. The following example answers this question in negative.

\begin{example} There exists a metrizable countable semitopological semilattice
 $X$ whose partial order is dense and non-closed in $X\times X$.
 \end{example}
 
 This example will be constructed in Section~\ref{s:Ex} after some preliminary work, made in Sections~\ref{s:T}--\ref{s:SL}. In Section~\ref{s:T} we establish necessary and sufficient conditions on a set $X$ and function $\ell:\dom(\ell)\to X$ defined on a subset $\dom(\ell)\subset X^\w$ ensuring that $X$ admits a (metrizable) Hausdorff topology  in which every sequence $s\in\dom(\ell)$ converges to the point $\ell(s)$. 
In Section~\ref{s:TA} we study the analogous problem for acts, i.e. sets endowed with monoids of self-maps and in Section~\ref{s:SG}, \ref{s:SL} we apply the obtained results about acts to constructing topologies with prescribed convergent sequences on  semigroups and semilattices.
More information on the closedness of the partial order in semitopological semilattices can be found in \cite[\S7]{BB}.

\section{Convergent sequences in topological spaces}\label{s:T}

Let $X$ be a set and $X^\w$ be its countable power. Elements of $X^\w$ are sequences $s=(s_n)_{n\in\w}$. Let $\ell:\dom(\ell)\to X$ be a function defined on a subset $\dom(\ell)\subset X^\w$. 

A topology $\tau$ on the set $X$ is called 
{\em $\ell$-admissible} if each sequence  $s\in\dom(\ell)$ converges to the point $\ell(s)$ in the topological space $(X,\tau)$.

Observe that the indiscrete topology $\{\emptyset,X\}$ on $X$ is $\ell$-admissible. So, the family of $\ell$-admissible topologies is not empty.
This family has the largest element. This is the topology $\tau_\ell$ consisting of all subsets $U\subset X$ such that for any sequence $s=(s_n)_{n\in\w}\in\dom(\ell)$ with $\ell(s)\in U$ the set $\{n\in\w:s_n\notin U\}$ is finite. The topology $\tau_\ell$ will be referred to as the {\em largest $\ell$-admissible topology} on  $X$. In this section we discuss the following problem.

\begin{problem} Under which conditions the largest $\ell$-admissible topology $\tau_\ell$ on $X$ is Hausdorff?
\end{problem}

Below we define two necessary conditions of the Haudorffness of the topology $\tau_\ell$.

The  function $\ell:\dom(\ell)\to X$ is defined to be  
\begin{itemize}
\item {\em $T_1$-separating} if for any sequence $s\in\dom(\ell)$ any any point $x\in X$ with $x\ne\ell(s)$ the set $\{n\in\w:s_n=x\}$ is finite;
\item {\em $T_2$-separating} if $\ell$ is $T_1$-separating and for any sequences $s,t\in\dom(\ell)$ with $\ell(s)\ne\ell(t)$ there exists a finite set $F\subset\w$ such that $s_n\ne t_m$ for any $n,m\in\w\setminus F$.
\end{itemize}

We say that a topology $\tau$ on a set $T$  {\em satisfies the separation axiom} $T_1$ if each finite subset of $T$ is $\tau$-closed in $T$. In this case we say that $(T,\tau)$ is a {\em $T_1$-space}.

\begin{lemma}\label{l:T1} The function $\ell$ is $T_1$-separating if and only if the topology $\tau_\ell$ satisfies the separation axiom $T_1$.
\end{lemma}

\begin{proof} Assume that $\ell$ is $T_1$-separating and take any finite set $F\subset X$. To show that the set $U:=X\setminus F$ belongs to the topology $\tau_\ell$, it suffices to check that for every $s\in \dom(\ell)$ with $\ell(s)\in U$ the set $\{n\in\w:s(n)\notin U\}$ is finite. By the $T_1$-separating property of $\ell$, for every $x\in F$ the set $\Omega_x=\{n\in\w:s_n=x\}$ is finite and so is the set $\{n\in\w:s_n\notin U\}=\bigcup_{x\in F}\Omega_x$.

Now assuming that each finite subset $F\subset X$ is closed in the topology $\tau_\ell$, we shall prove that the function $\ell$ is $T_1$-separating. Given any point $x\in X\setminus\{\ell(s)\}$, observe that $\ell(s)\in X\setminus\{x\}\in\tau_\ell$ implies that the set $\{n\in\w:s_n\notin X\setminus\{x\}\}=\{n\in\w:s_n=x\}$ is finite, which means that $\ell$ is $T_1$-separating.
\end{proof}

\begin{lemma}\label{l:T2a} If the topology $\tau_\ell$ is Hausdorff, then the function $\ell$ is $T_2$-separating.
\end{lemma}

\begin{proof} By Lemma~\ref{l:T1}, the function $\ell$ is $T_1$-separating. To prove that $\ell$ is $T_2$-separating, take two sequences $s,t\in \dom(\ell)$ with $\ell(s)\ne\ell(t)$. By the Hausdorff property of the topology $\tau_\ell$, there are disjoint open sets $U,V\in\tau_\ell$ such that $\ell(s)\in U$ and $\ell(t)\in V$. By the definition of the topology $\tau_\lambda$, the sets $F:=\{n\in\w:s_n\notin U\}$ and $E:=\{m\in\w:t_m\notin V\}$ are finite. Then $s_n\ne t_m$ for any $n,m\in\w\setminus(F\cup E)$.
\end{proof}

Now we shall prove that the largest $\ell$-admissible topology $\tau_\ell$ is Hausdorff if the function $\ell$ is $T_2$-separating and $\dom(\ell)$ is at most countable. 

For a sequence $s\in\dom(\ell)$ and a subset $I\subset\w$ let $s[I]:=\{s_n:n\in I\}$ and $s[I]^*:=s[I]\cup\{\ell(s)\}$.

\begin{lemma}\label{l:T2b} Assume that the function $\ell:\dom(\ell)\to X$ is $T_2$-separating. Let $s\in\dom(\ell)$ and $I\subset\w$.
\begin{enumerate}
\item the set $s[I]^*$ is closed in $(X,\tau_\ell)$;
\item each point $x\in s[I]^*\setminus\{\ell(s)\}$ is isolated in $s[I]^*$;
\item the subspace $s[I]^*$ of $(X,\tau_\ell)$ is compact and Hausdorff.
\end{enumerate}
\end{lemma}

\begin{proof} 1. The inclusion $X\setminus s[I]^*\in \tau_\ell$ will follow as soon as 
we show that for any sequence $t\in \dom(\ell)$ with $\ell(t)\notin s[I]^*$ the set $\{n\in\w:t_n\in s[I]^*\}$ is finite.  By the $T_2$-separating property of $\ell$, there exists a finite set $\Omega\subset\w$ such that $s_n\ne t_m$ for any $n,m\in\w\setminus\Omega$.  Consider the finite set $E=\{\ell(s)\}\cup\{s_n:n\in\Omega\}\setminus\{\ell(t)\}$. By  the $T_1$-separating property of $\ell$, the set $\Lambda=\Omega\cup\{n\in\w:t_n\in E\}$ is finite. Then the set $\{n\in\w:t_n\in s[I]^*\}\subset \Lambda$ is finite, too. 
\smallskip

2. Given any point $x\in s[I]^*\setminus\{\ell(s)\}$, observe that $s[I]^*\setminus\{x\}=s[J]^*$ where $J=I\setminus s^{-1}(x)$.

By Lemma~\ref{l:T2b}(1), the subspace $s[J]^*$ is closed in $(X,\tau_\ell)$ and then the singleton $\{x\}=s[I]^*\setminus s[J]^*$ is open in $s[I]^*$.
\smallskip

3. The compactness of $s[I]^*$ follows from the fact that each neighborhood $U\in\tau_\ell$ of $\ell(s)$ contains all but finitely many points of the set $s[I]$. To see that $s[I]^*$ is Hausdorff, take any two distinct points $x,y\in s[I]^*$. One of these points is distinct from the limit point $\ell(s)$ of the sequence $s$ and we lose no generality assuming that $x\ne\ell(s)$. By Lemmas~\ref{l:T1} and \ref{l:T2b}(2), the singleton $U_x=\{x\}$  closed-and-open in $s[I]^*$ and so is its complement $U_y=s[I]^*\setminus U_x$. Then $U_x,U_y$ are disjoint neighborhoods of the points $x,y$ in $s[I]^*$, witnessing that the subspace $s[I]^*$ of $(X,\tau_\ell)$ is Hausdorff.
\end{proof}

\begin{corollary}\label{c:T2} If the function $\ell:\dom(\ell)\to X$ is $T_2$-separating, then for every finite subset $S\subset \dom(\lambda)$ the subspace $S[\w]^*=\bigcup_{s\in S}s[\w]^*$ of $(X,\tau_\ell)$ is compact, Hausdorff, and closed in $(X,\tau_\ell)$.
\end{corollary}

This corollary follows from Lemma~\ref{l:T2b} and the following known fact.

\begin{lemma} The union $A\cup B$ of two closed Hausdorff subspaces of a topological space $T$ is Hausdorff.
\end{lemma}

\begin{proof} We lose no generality assuming that $T=A\cup B$. The Hausdorff property of the space $T=A\cup B$ will follow as soon as we check that its diagonal $\Delta_T=\{(x,x):x\in T\}$ is closed in $T\times T$. For this observe that $\Delta_T=\Delta_A\cup\Delta_B$. Since the space $A$ is Hausdorff and closed in $T$, its diagonal $\Delta_A$ is closed in $A\times A$ and in $T\times T$. By analogy, the diagonal $\Delta_B$ is closed in $B\times B$ and in $T\times T$. Then the union $\Delta_T=\Delta_A\cup\Delta_B$ is closed in $T\times T$ and the space $T$ is Hausdorff.
\end{proof}

We say that the topology $\tau$ of a topological space $X$ is {\em generated} by a family $\K$ of subspaces of $X$ if a set $U\subset X$ is open in $X$ if any only if for every $K\in\K$ the intersection $U\cap K$ is open in the subspace topology of $K$.
A topology $\tau$ on a set $X$ is called {\em a $k_\w$-topology} if it is generated by a countable family of compact subsets of the topological space $(X,\tau)$.

\begin{lemma}\label{l:T2n} If the function $\ell:\dom(\ell)\to X$ is $T_2$-separating and $\dom(\ell)$ is at most countable, then the topology $\tau_\ell$ is Hausdorff and normal. Moreover, $\tau_\lambda$ is a $k_\w$-topology, generated by the countable family $\K=\{s[\w]^*:s\in\dom(\ell)\}$ of compact sets in $(X,\tau_\ell)$.
\end{lemma}

\begin{proof} The definition of the topology $\tau_\ell$ ensures that it is generated by the countable family $\K$. Now we show that the topology $\tau_\ell$ is Hausdorff and normal. By Lemma~\ref{l:T1}, the topology $\tau_\ell$ satisfies the separation axiom $T_1$. Now it suffices to check that this topology is normal. From now on, we consider $X$ as a topological space endowed with the topology $\tau_\ell$. 

Given two disjoint closed sets $A,B\subset X$ we should find two disjoint open sets $V,W\subset X$ such that $A\subset V$ and $B\subset W$.

 Let $\dom(\ell)=\{s_n\}_{n\in\w}$ be an enumeration of the countable set $\dom(\ell)$. By Corollary~\ref{c:T2}, for every $n\in\w$ the subspace $K_n:=\bigcup_{i\le n}s_i[\w]^*$ of $(X,\tau_\ell)$ is compact, Hausdorff, and closed in $(X,\tau_\ell)$. Let $A_0:=A\cap K_0$ and $B_0:=B\cap K_0$. By  induction we shall construct sequences $(A_n)_{n\in\w}$, $(B_n)_{n\in\w}$, $(V_n)_{n\in\w}$, $(W_n)_{n\in\w}$ of subsets in $X$ such that for every $n\in\w$ the following conditions are satisfied:
\begin{enumerate}
\item the sets $A_n,B_n$ are disjoint and closed in $K_n$;
\item $A_n\subset V_n\subset K_n$ and $B\subset W_n\subset K_n$;
\item the sets $V_n,W_n$ are open in $K_n$ and $\overline V_n\cap\overline W_n=\emptyset$;
\item $A_{n+1}=\overline{V}_n\cup(K_{n+1}\cap A)$ and  $B_{n+1}=\overline W_n\cup(K_{n+1}\cap B)$;
\item $A_{n+1}\cap B_{n+1}=\emptyset$. 
\end{enumerate} 
Assume that for some $n\in\w$ disjoint closed sets $A_n,B_n\subset K_n$ with $K_n\cap A\subset A_n$ and $K_n\cap B\subset B_n$ have been constructed. By the normality of the compact Hausdorff space $K_n$, there are open sets $V_n,W_n\subset K_n$ satisfying the conditions (2),(3). Define the sets $A_{n+1},B_{n+1}$ by the formula (4) and observe that 
$$A_{n+1}\cap B_{n+1}=(A_{n+1}\cap B_{n+1}\cap K_n)\cup(A_{n+1}\cap B_{n+1}\cap(K_{n+1}\setminus K_n)\subset(\overline V_n\cap\overline W_n)\cup(A\cap B)=\emptyset.$$
After completing the inductive construction, observe that $V:=\bigcup_{n\in\w}V_n$ and $W=\bigcup_{n\in\w}W_n$ are disjoint open sets in $X$ such that $A\subset V$ and $B\subset W$.
\end{proof}

The following theorem is the main result of this section.

\begin{theorem}\label{t:T} For a set $X$ and function $\ell:\dom(\ell)\to X$ defined on a countable subset $\dom(\ell)\subset X^\w$ the following conditions are equivalent:
\begin{enumerate}
\item $X$ admits a metrizable topology $\tau$ in which every sequence $s\in\dom(\ell)$ converges to the point $\ell(s)$.
\item $X$ admits a Hausdorff topology $\tau$ in which every sequence $s\in\dom(\ell)$ converges to the point $\ell(s)$.
\item The following two properties are satisfied:
\begin{itemize}
\item[$(3a)$] for any $s\in\dom(\ell)$ and $x\in X$ with $s\ne\ell(s)$ the set $\{n\in\w:s_n=x\}$ is finite;
\item[$(3b)$] for any sequences $s,t\in\dom(\ell)$ with $\ell(s)\ne\ell(t)$ there exists a finite set $F\subset\w$ such that $s_n\ne t_m$ for any $n,m\in\w\setminus F$.
\end{itemize}
\end{enumerate}
\end{theorem}

\begin{proof} The implications $(1)\Ra(2)\Ra(3)$ are trivial.

To prove that $(3)\Ra(1)$, assume the the condition (3) is satisfied.
Then the function $\ell$ is $T_2$-separating and by Lemma~\ref{l:T2n}, the largest $\ell$-admissible topology $\tau_\ell$ is Hausdorff and normal. Consider the countable subset $D=\bigcup_{s\in\dom(\ell)}s[\w]^*$ of $X$ and observe that $X\setminus D$ is a closed-and-open discrete subspace of the topological space $X$ endowed with the topology $\tau_\ell$. Being countable and Tychonoff, the closed-and-open subspace $D$ of $X$ is zero-dimensional. Then for any distinct points $x,y\in D$ we can choose a closed-and-open subset $U_{x,y}\subset X$ such that $x\in U_{x,y}$ and $y\notin U_{x,y}$. Let $\tau$ be the topology on $D$, generated by the countable subbase $\{U_{x,y},X\setminus U_{x,y}:x,y\in D,\;x\ne y\}$. It is clear that the topology $\tau$ is second-countable, Hausdorff, zero-dimensional and hence regular. By Urysohn Metrization Theorem~\cite[4.2.9]{Eng}, the topological space $D_\tau=(D,\tau)$ is metrizable. Then the topology of the topological sum on $(X\setminus D)\oplus D_\tau$ is also metrizable. Since $\tau\subset\tau_\ell$, the topology $\tau$ is $\ell$-admissible, which means that each sequence $s\in\dom(\ell)$ converges to the point $\ell(s)$. 
\end{proof}

The countability of the domain $\dom(\ell)$ in Theorems~\ref{t:T} is essential as shown by the following example.

\begin{example} There exists a  $T_2$-separating function $\ell:\dom(\ell)\to \{0,1\}\subset \w$ defined on a subset $\dom(\ell)\subset[\w]^\w$ of cardinality $|\dom(\ell)|=\w_1$ such that the largest $\ell$-admissible topology $\tau_\ell$ is not Hausdorff.
\end{example}

\begin{proof} To construct such function $\ell$, take any Hausdorff $(\w_1,\w_1)$-gap on $\w$, which is a pair $\big((A_i)_{i\in\w_1},(B_i)_{i\in\w_1}\big)$ of families of infinite subsets of $\w$ satisfying the following two conditions:
\begin{itemize}
\item[(H1)] for any $i<j<\w_1$ we have $A_i\subset^*A_j$ and $B_i\subset^* B_j$;
\item[(H2)] $A_i\cap B_j$ is finite for any $i,j\in\w_1$;
\item[(H3)] for any set $C\subset\w$ one of the sets $\{i\in\w_1:A_i\subset^* C\}$ or $\{i\in\w_1:B_i\subset^*\w\setminus C\}$ is at most countable.
\end{itemize}

Here the notation $A\subset^* B$ means that the complement $A\setminus B$ is finite. 

It is well-known \cite[Ch.20]{JW2} that Hausdorff $(\w_1,\w_1)$-gaps do exist in ZFC.

For every $i\in\w_1$ choose any bijective functions $\alpha_i:\w\to A_i$ and $\beta_i:\w\to B_i$, and put $\dom(\ell)=\{\alpha_i,\beta_i:i\in\w_1\}$. Let $\ell:\dom(\ell)\to\{0,1\}\subset\w$ be the function such that $\ell^{-1}(0)=\{\alpha_i\}_{i\in\w_1}$ and $\ell^{-1}(1)=\{\beta_i\}_{i\in\w}$. The injectivity of the functions $\alpha_i,\beta_i$ and the condition (H2) ensure that the function $\ell$ is $T_2$-separating. Assuming that the $\ell$-admissble topology on $\w$ is Hausdorff, we could find two disjoint open sets $U_0,U_1\in\tau_\ell$ such that $0\in U_0$ and $1\in U_1$. By condition (H3), there exists $i\in\w_1$ such that $A_i\not\subset^* U_0$ or $B_i\not\subset^* U_1$. In the first case the set $\{n\in\w:\alpha_i(n)\notin U_0\}$ is infinite, which contradicts $U_0\in\tau_\ell$. In the second case  the set $\{n\in\w:\beta_i(n)\notin U_1\}$ is infinite, which contradicts $U_1\in\tau_\ell$.
\end{proof}

\section{Convergent sequences in topological acts}\label{s:TA}

For a set $X$ denote by $X^X$ the set of all self-maps $X\to X$. The set $X^X$ endowed with the operation of composition is a monoid whose unit is the identity map $\id_X$ of $X$. 

An {\em act} is a pair $(X,\A)$ consisting of a set $X$ and a submonoid $\A\subset X^X$. Elements of the set $\A$ are called the {\em shifts} of the act $(X,\A)$.


A topology $\tau$ on the underlying set $X$ of an act $(X,\A)$ is called an {\em shift-continuous} if each shift $\alpha\in \A$ is a continuous self-map of the topological space $(X,\tau)$. 

\begin{theorem}\label{t:TA}  For an act $(X,\A)$ with countable set $\A$ of shifts and  a function $\ell:\dom(\ell)\to X$ defined on a countable subset $\dom(\ell)\subset X^\w$ the following conditions are equivalent:
\begin{enumerate}
\item $X$ admits a shift-continuous metrizable topology $\tau$ in which every sequence $s\in\dom(\ell)$ converges to the point $\ell(s)$;
\item $X$ admits a shift-continuous Hausdorff topology $\tau$ in which every sequence $s\in\dom(\ell)$ converges to the point $\ell(s)$;
\item the following two properties hold:
\begin{itemize}
\item[(3a)] for any $s\in\dom(\ell)$, $\alpha\in \A$ and $x\in X$ with $x\ne \alpha\circ \ell(s)$ the set $\{n\in\w:\alpha\circ s_n=x\}$ is finite;
\item[(3b)] for any sequences $s,t\in\dom(\ell)$ and any shifts $\alpha,\beta\in\A$ with $\alpha{\circ}\ell(s)\ne\beta{\circ}\ell(t)$ there exists a finite set $F\subset \w$ such that $\alpha{\circ}s_n\ne \beta{\circ}t_m$ for any $n,m\in\w\setminus F$.
\end{itemize}
\end{enumerate}
\end{theorem}

\begin{proof} The implications $(1)\Ra(2)\Ra(3)$ are trivial.
To prove that $(3)\Ra(1)$, assume that the condition (3) is satisfied. Consider the set $\dom(\A\ell)=\{\alpha\circ s:\alpha\in \A,\;s\in \dom(\ell)\}\subset X^\w$ and the function $\A\ell:\dom(\A\ell)\to X$ defined by the formula $\A\ell(\alpha\circ s)=\alpha(\ell(s))$. The conditions (3a) and (3b) ensure that the function $\A\ell$ is well-defined and is $T_2$-separating. By Lemma~\ref{l:T2n}, the largest $\A\ell$-admissible topology $\tau_{\A\ell}$ on $X$ is Hausdorff and normal.  Since the monoid $\A$ contains the identity map of $X$, the topology $\tau_{\A\ell}$ is $\ell$-admissible, which implies that each sequence $s\in\dom(\ell)$ converges to $\ell(s)$ in the topological space $(X,\tau_{\A\ell})$.

We claim that the topology $\tau_{\A\ell}$ is shift-continuous.  Given any shift $\alpha\in\A$ and open set $U\in\tau_{\A\ell}$, we need to check that $\alpha^{-1}(U)\in\tau_{\A\ell}$. The latter inclusion holds if any only if for any sequence $s\in\dom(\A\ell)$ with $\A\ell(s)\in \alpha^{-1}(U)$ the set $\{n\in\w:s_n\notin \alpha^{-1}(U)\}=\{n\in\w:\alpha\circ s_n\notin U\}$ is finite. Since $\A$ is a monoid, the sequence $\alpha\circ s$ belongs to $\dom(\A\ell)$. Since $U\in\tau_{\A\ell}$ and $\A\ell(\alpha\circ s)=\alpha(\A\ell(s))\in U$, the set   $\{n\in\w:\alpha\circ s_n\notin U\}=\{n\in\w:s_n\notin \alpha^{-1}(U)\}$ is finite and we are done.
\smallskip

 The countability of the sets $\dom(\ell)$ and $\A$ imply the countability of the sets $\dom(\A\ell)$ and $D=\{s[\w]^*:s\in \dom(\A\ell)\}$.
Observe that for every $\alpha\in \A$ we have $\alpha[D]\subset D$.

By the definition of the topology $\tau_{\A\ell}$, the set $D$ is open-and-closed in $(X,\tau_{\A\ell})$ and the complement $X\setminus D$ is discrete. Being Tychonoff, the countable subspace $D$ of $(X,\tau_{\A\ell})$ is zero-dimensional. This allows us to choose a countable family $\mathcal B\subset\tau_{\A\ell}$ of open-and-closed sets that separate points of the countable set $D$ in the sense that for any distinct points $x,y\in D$ there exists a set $B\in\mathcal B$ such that $x\in B$ and $y\notin B$. Since $\tau_{\A\ell}$ is an act topology on $(X,\A)$, for every $\alpha\in \A$ and $B\in\mathcal B$ the set $\alpha^{-1}(B)$ is closed-and-open. Then the topology $\tau_D$ on $D$ generated by the subbase $\{D\cap \alpha^{-1}(B), D\setminus \alpha^{-1}(B):\alpha\in\A,\;B\in\mathcal B\}$ is second-countable, Hausdorff and zero-dimensional. By the Urysohn metrization Theorem \cite[4.2.9]{Eng}, the topological space $D_\tau=(D,\tau_D)$ is metrizable. Then the topology $\tau$ of topological sum $D_\tau\oplus(X\setminus D)$ of $D_\tau$ and the discrete topological space $X\setminus D$ is metrizable. By the definition of the topology $\tau_D$, for every $\alpha\in\A$ the restriction $\alpha{\restriction}D$ is a continuous self-map of the topological space $D_\tau$. Since $X\setminus D$ is a closed-and-open discrete subspace of $(X,\tau)$, the continuity of $\alpha{\restriction}D$ implies that $\alpha$ is a continuous self-map of the metrizable topological space $(X,\tau)$. This means that the topology $\tau$ is shift-continuous. Since $\tau\subset\tau_{\A\ell}$, the metrizable topology $\tau$ is $\ell$-admissible.
\end{proof}

Theorem~\ref{t:TA} can be compared with the following result proved in \cite[3.4]{BPS}.

\begin{theorem}[Banakh, Protasov, Sipacheva] Let $\kappa$ be an infinite cardinal, $(X,\A)$ be an act, $x\in X$, and $(x_i)_{i\in \kappa}$ be a transfinite sequence of points in $X$. Assume that there exists a (not necessarily bijective) enumeration $\A=\{\alpha_i\}_{i\in\kappa}$ of the set $\A$ such that for each ordinal $m\in\kappa$ and ordinals $i,j,k<m$ the following conditions are satisfied:
\begin{enumerate}
\item if $\alpha_i(x)\ne \alpha_j(x)$, then $\alpha_i(x_m)\ne \alpha_j(x_m)$;
\item if $\alpha_i(x)\ne \alpha_j(x_k)$, then $\alpha_i(x_m)\ne \alpha_j(x_k)$.
\end{enumerate}
Then $X$ admits a shift-continuous hereditarily normal topology $\tau$ in which the transfinite sequence $(x_i)_{i\in\lambda}$ converges to the point $x$ in the sense that
for every neighborhood $O_x\in\tau$ of $x$ there exists $n\in\kappa$ such that $x_i\in O_x$ for all $i\ge n$ in $\kappa$.
\end{theorem}

\section{Convergent sequences in semigroups}\label{s:SG}

Let $X$ be a semigroup and $X^1$ be the semigroup $X$ with attached unit. A topology $\tau$ on $X$ is called {\em shift-continuous} if for every $a,b\in X^1$ the two-sided shift $$X\to X,\;\;x\mapsto axb,$$is a continuous self-map of the topological space $(X,\tau)$. 


Each semigroup $X$ has the structure of an act $(X,\A)$ endowed with the family of shifts $\A=\{s_{a,b}:a,b\in X^1\}$. Applying Theorem~\ref{t:TA} to this act, we obtain the following theorem, which is a main result of this section.

\begin{theorem}\label{t:TS}  For a countable semigroup $X$ and  a function $\ell:\dom(\ell)\to X$ defined on a countable subset $\dom(\ell)\subset X^\w$, the following conditions are equivalent:
\begin{enumerate}
\item The semigroup $X$ admits a shift-continuous metrizable topology $\tau$ in which every sequence $s\in\dom(\ell)$ converges to the point $\ell(s)$;
\item The semigroup $X$ admits a shift-continuous Hausdorff topology $\tau$ in which every sequence $s\in\dom(\ell)$ converges to the point $\ell(s)$;
\item the following two properties hold:
\begin{itemize}
\item[(3a)] for any $s\in\dom(\ell)$, $a,b\in X^1$ and $x\in X$ with $x\ne a{\cdot}\ell(s){\cdot}b$ the set $\{n\in\w:a{\cdot}s_n{\cdot}b=x\}$ is finite;
\item[(3b)] for any $s,t\in\dom(\ell)$ and $a,b,c,d\in X^1$ with $a{\cdot}\ell(s){\cdot}b\ne c{\cdot}\ell(t){\cdot}d$ there exists a finite set $F\subset\w$ such that $a{\cdot}s_n{\cdot}b\ne c{\cdot}t_m{\cdot}d$ for all $n,m\in\w\setminus F$.
\end{itemize}
\end{enumerate}
\end{theorem}

For commutative semigroups, Theorem~\ref{t:TS} has a bit simpler form.

\begin{theorem}\label{t:TAS}  For a countable commutative semigroup $X$ and  a function $\ell:\dom(\ell)\to X$ defined on a countable subset $\dom(\ell)\subset X^\w$, the following conditions are equivalent:
\begin{enumerate}
\item The semigroup $X$ admits a shift-continuous metrizable topology $\tau$ in which every sequence $s\in\dom(\ell)$ converges to the point $\ell(s)$.
\item The semigroup $X$ admits a shift-continuous Hausdorff topology $\tau$ in which every sequence $s\in\dom(\ell)$ converges to the point $\ell(s)$.
\item The following two properties hold:
\begin{itemize}
\item[(3a)] for any $s\in\dom(\ell)$, $a\in X^1$ and $x\in X$ with $x\ne a{\cdot}\ell(s)$ the set $\{n\in\w:a{\cdot}s_n=x\}$ is finite;
\item[(3b)] for any $s,t\in\dom(\ell)$ and $a,b\in X^1$ with $a{\cdot}\ell(s)\ne b{\cdot}\ell(t)$ there exists a finite subset $F\subset\w$ such that $a{\cdot}s_n\ne b{\cdot}t_m$ for any $n,m\in\w\setminus F$.
\end{itemize}
\end{enumerate}
\end{theorem}

\section{Convergent sequences in semilattices}\label{s:SL}

Applying Theorem~\ref{t:TAS} to semilattices we obtain the following characterization.

\begin{theorem}\label{t:TL}  For a countable semilattice $X$ and  a function $\ell:\dom(\ell)\to X$ defined on a countable subset $\dom(\ell)\subset X^\w$ the following conditions are equivalent:
\begin{enumerate}
\item The semilattice $X$ admits a shift-continuous metrizable topology $\tau$ in which every sequence $s\in\dom(\ell)$ converges to the point $\ell(s)$.
\item The semilattice $X$ admits a shift-continuous Hausdorff topology $\tau$ in which every sequence $s\in\dom(\ell)$ converges to the point $\ell(s)$.
\item The following two conditions hold:
\begin{itemize}
\item[(3a)] for any $s\in\dom(\ell)$, $a\in X$ and $x\in X$ with $x\ne a{\cdot}\ell(s)$ the set $\{n\in\w:a{\cdot}s_n=x\}$ is finite;
\item[(3b)] for any $s,t\in\dom(\ell)$ and $a,b\in X$ with $a{\cdot}\ell(s)\ne b{\cdot}\ell(t)$ there exists a finite set $F\subset\w$ such that $a{\cdot}s_n\ne b{\cdot}t_m$ for any $n,m\in\w\setminus F$.
\end{itemize}
\end{enumerate}
\end{theorem}
 
\begin{proof} The implications $(1)\Ra(2)\Ra(3)$ are trivial. The implication $(3)\Ra(1)$ will follow from Theorem~\ref{t:TAS} as soon as we check that the condition $(3)$ of Theorem~\ref{t:TL} implies condition (3) of Theorem~\ref{t:TAS}. So, assume that condition (3) of Theorem~\ref{t:TL} is satisfied.

To check the condition (3a) of Theorem~\ref{t:TAS}, take any sequence $s\in \dom(\ell)$ and element $a\in X^1$ and $x\in X$ with $x\ne a{\cdot}\ell(s)$. If $a\in X$, then the set $\{n\in\w:a{\cdot}s_n=x\}$ is finite by the condition (3a) of Theorem~\ref{t:TL}. So, we assume that $a$ is an external unit for $X$. In this case $\{n\in\w:a{\cdot}s_n=x\}=\{n\in\w:s_n=x\}\subset\{n\in\w:x{\cdot}s_n=x{\cdot}x=x\}$. Assuming that the set $\{n\in\w:s_n=x\}$ is infinite, we conclude that the set $\{n\in\w:x{\cdot}s_n=x\}$ is infinite, which implies that $x{\cdot}\ell(s)=x$. Since $\ell(s){\cdot}\ell(s)=\ell(s)=a{\cdot}\ell(s)\ne x=\ell(s){\cdot}x$, the set
$$\{n\in\w:\ell(s){\cdot}s_n=x\}=\{n\in\w:\ell(s){\cdot}s_n=\ell(s){\cdot}x\}\supset \{n\in\w:s_n=x\}$$is finite, which contradicts our assumption.

Next, we check the condition (3b) of Theorem~\ref{t:TAS}. Given any sequences $s,t\in\dom(\ell)$ and elements $a,b\in X^1$ with $a{\cdot}\ell(s)\ne b{\cdot}\ell(t)$, we need to find a finite set $F\subset\w$ such that $a{\cdot}s_n\ne b{\cdot}t_m$ for any $n,m\in\w\setminus F$. 
The condition (3a) of Theorem~\ref{t:TL} ensures that $a$ or $b$ does not belong to $X$. We lose no generality assuming that $a\notin X$ and hence $a$ is the external unit to $X$. In this case the inequality  $a{\cdot}\ell(s)\ne b{\cdot}\ell(t)$ transforms into the inequality $\ell(s)\ne b{\cdot}\ell(t)$. We claim that there exists an element $c\in X$ such that $c{\cdot}\ell(s)\ne cb{\cdot}\ell(t)$. If $\ell(s)\ne \ell(s){\cdot}b{\cdot}\ell(t)$, then put $c=\ell(s)$. If $\ell(s)=\ell(s){\cdot}b{\cdot}\ell(t)$, then put $c=b{\cdot}\ell(t)$ and conclude that $c{\cdot}\ell(s)=\ell(s)\ne b{\cdot}\ell(t)=cb{\cdot}\ell(t)$. 

In both cases we get $c{\cdot}\ell(s)\ne cb{\cdot}\ell(t)$. By condiction (3b) of Theorem~\ref{t:TL}, there exists a finite set $F\subset \w$ such that $c{\cdot}s_n\ne cb{\cdot}t_m$ and hence $s_n\ne b{\cdot}t_m$ for any $n,m\in\w\setminus F$.
\end{proof}

A topology on a semilattice $X$ is called {\em Lawson} if it has a base consisting of open subsemilattices.

\begin{example} There exists a countable semilattice $X$ and a function $\ell:\{s,t\}\to X$ defined on a subset $\{s,t\}\subset X^\w$ such that
\begin{enumerate}
\item The semilattice $X$ admits a shift-continuous metrizable topology $\tau$ in which the sequence $s$ converges to $\ell(s)$ and the sequence $t$ converges to $\ell(t)$.
\item The semilattice $X$ admits no Lawson Hausdorff topology $\tau$ in which the sequence $s$ converges to $\ell(s)$ and the sequence $t$ converges to $\ell(t)$.
\end{enumerate}
\end{example}

\begin{proof} By \cite[2.21]{CHK}, there exists a compact metrizable topological semilattice $K$ containing two points $x,y\in K$ such that for any neighborhoods $O_x,O_y\subset K$ of the points $x,y$ there exists a finite subset $F\subset O_x$ such that $\inf F\in O_y$. Fix countable neighborhood bases $\{V_n\}_{n\in\w}$ and $\{W_n\}_{n\in\w}$ at the points $x,y$, respectively. For every $n\in\w$ choose a finite subset $F_n\subset \bigcap_{k\le n}V_k$ such that $\inf F_n\in \bigcap_{k\le n}W_k$. Let $s\in X^\w$ be a sequence such that $s_n=\inf F_n$ for every $n\in\w$ and $t\in X^\w$ be a sequence such that $F_n=\{t_k:\sum_{i<n}|F_i|<k\le\sum_{i\le n}|F_i|\}$ for every $n\in\w$. Let $\ell(s)=y$ and $\ell(t)=x$. 

Let $X$ be the countable semilattice generated by the countable set $\{x,y\}\cup s[\w]\cup t[\w]$. The metrizable topology on $X$ inherited from $K$ witnesses that the condition (1) is satisfied.

It remains to prove that $X$ admits no Lawson  Hausdorff topology $\tau$ in which the sequence $s$ converges to $\ell(s)=y$ and the sequence $t$ converges to $\ell(t)=x$. To derive a contradiction, assume that such topology $\tau$ exists. Then the points $x,y$ have disjoint open neighborhoods $O_x,O_y\in\tau$ such that $O_x$ is a subsemilattice of $X$. Since the sequences $s$ and $t$ converge to $y$ and $x$, respectively, there exists $n\in\w$ such that $s_k\in O_y$ and $t_k\in O_x$ for all $k\ge n$. Then $F_n\subset\{t_k\}_{k\ge n}\subset O_x$ and hence $\inf F_n\in O_x$ (as $O_x$ is a subsemilattice of $X$). On the other hand, $\inf F_n=s_n\in O_y$.  Then   $\inf F_n\in O_x\cap O_y$, which contradicts the choice of the neighborhoods $O_x$ and $O_y$.
\end{proof}

\section{The example}\label{s:Ex}

In this section we shall apply Theorem~\ref{t:TL} to construct an example of a metrizable semitopological semilattice with dense non-closed partial order.

Consider the semilattice $\{0,1,2\}$ endowed with the operation of taking minimum. In the semilattice $\{0,1,2\}^\w$ consider the countable subsemilattice $X$ consisting of functions $f:\w\to\{0,1,2\}$ having non-empty finite support $\supp(f):=f^{-1}(\{0,1\})$. 

It is easy to see that the partial order on $\{0,1,2\}$ induced by the semilattice operation (of minimum) coincides with the usual linear order on $\{0,1,2\}$. Then the semilattice operation (of coordinatewise minimum) on $X\subset \{0,1,2\}^\w$ induces the natural partial order on $X$.

For every $n\in\w$ consider the functions $\mathbf 0_n,\mathbf 1_n\in X$ defined by $$\mathbf 0_n(i)=\begin{cases}0&\mbox{if $i=n$}\\
2&\mbox{otherwise}
\end{cases}\quad \mbox{ and }\quad
\mathbf 1_n(i)=\begin{cases}1&\mbox{if $i=n$}\\
2&\mbox{otherwise}.
\end{cases}
$$
It is clear that $\mathbf 0_n{\cdot}\mathbf 1_n=\mathbf 0_n$ and hence $\mathbf 0_n\le \mathbf 1_n$ for all $n\in\w$.

\begin{theorem}\label{t:main} The semilattice $X$ admits a metrizable  shift-continuous topology $\tau$ such that the set $\{(\mathbf 0_n,\mathbf 1_n):n\in\mathbb N\}$ is dense in the square $X\times X$ of the semitopological semilattice $(X,\tau)$. Since $$\{(\mathbf 0_n,\mathbf 1_n)\}_{n\in\w}\subset P:=\{(x,y)\in X\times X:xy=x\}\ne X\times X$$ the partial order $P$ of $X$ is dense and non-closed in $X\times X$.
\end{theorem}

\begin{proof} Write the countable set $X\times X$ as $X\times X=\{(x_k,y_k):k\in\w\}$. For every $k,n\in\w$ consider the elements $z_{k,n}=\mathbf 0_{2^k3^n}$ and $u_{k,n}:=\mathbf 1_{2^k3^n}$ of the set $X$. These elements form sequences $\vec z_k=(z_{k,n})_{n\in\w}$ and $\vec u_k=(u_{k,n})_{n\in\w}$, which are elements of the set $X^\w$. Let $\dom(\ell)=\{\vec z_k,\vec u_k:k\in\w\}$ and $\ell:\dom(\ell)\to X$ be the function defined by $\ell(\vec z_k)=x_k$ and $\ell(\vec u_k)=y_k$ for  $k\in\w$.

We claim that for the function $\ell$ the condition (3) of Theorem~\ref{t:TL} is satisfied. In fact, the condition (3a) is satisfied in the stronger form: for any $s\in \dom(\ell)$ and $a,b\in X$ the set $\{n\in\w:a{\cdot} s_n=b\}\subset \{n\in\w:\exists k\in\w$ with $2^k3^n\in \supp(b)\}$ is finite.

Now we check the condition (3b). Fix elements $a,b\in X$ and sequences $s,t\in \dom(\lambda)$ such that $a{\cdot}\ell(s)\ne b{\cdot}\ell(t)$. It is easy to see that  $a{\cdot}s_n\ne b{\cdot}t_m$ for any $n,m\in\w\setminus F$ where $F=\{n\in\w:\exists k\in\w$ such that $2^k3^n\in \supp(a)\cup\supp(b)\}$.

By Theorem~\ref{t:TL}, the semilattice $X$ admits a metrizable shift-continuous topology $\tau$ in which every sequence $s\in\dom(\ell)$ converges to $\ell(s)$. In particular, for every $k\in\w$ the sequence $(\mathbf 0_{2^k3^n})_{n\in\w}=\vec z_k$ converges to $x_k=\ell(\vec z_k)$ and the sequence $(\mathbf 1_{2^k3^n})_{n\in\w}=\vec z_k$ converges to $y_k=\ell(\vec u_k)$. Consequently, the set $\{(\mathbf 0_{2^k3^n},\mathbf 1_{2^k3^n}):k,n\in\w\}\subset \{(\mathbf 0_m,\mathbf 1_m):m\in\w\}$ is dense in $X\times X=\{(x_k,y_k):k\in\w\}$. 

Since $$\{(\mathbf 0_n,\mathbf 1_n)\}_{n\in\w}\subset P:=\{(x,y)\in X\times X:xy=x\}\ne X\times X,$$ the partial order $P$ is a dense non-closed subset of $X\times X$. 
\end{proof}

\begin{problem} Does the semilattice $X$ admit a Lawson Hausdorff shift-continuous topology such that the partial order $P=\{(x,y)\in X\times X:xy=x\}$ is not closed (and dense) in $X\times X$?
\end{problem}

\end{document}